\documentclass[12pt]{amsart}
\usepackage{amsfonts, amsbsy, amsmath, amssymb}

\newtheorem{thm}{Theorem}[section]
\newtheorem{lem}[thm]{Lemma}
\newtheorem{cor}[thm]{Corollary}

\newtheorem{rmk}[thm]{Remark}

\newtheorem{thm-con}[thm]{Theorem-Conjecture}
\numberwithin{equation}{section}

\theoremstyle{definition}
\newtheorem{defn}[thm]{Definition}

\newcommand{\f}{\Bbb F}

\begin{document}

\title[Self-Reciprocal Polynomials and Coterm Polynomials]{Self-Reciprocal Polynomials and Coterm Polynomials}

\author[Neranga Fernando]{Neranga Fernando}
\address{Department of Mathematics,
Northeastern University, Boston, MA 02115}
\email{w.fernando@neu.edu}

\begin{abstract}
We classify all self-reciprocal polynomials arising from reversed Dickson polynomials over $\mathbb{Z}$ and $\f_p$, where $p$ is prime. As a consequence, we also obtain coterm polynomials arising from reversed Dickson polynomials. 
\end{abstract}



\maketitle

%%%%%%%%%%%%%%%%%%%%%%%%%%%%%%%%%%%%%%%%
%  section 1 
%%%%%%%%%%%%%%%%%%%%%%%%%%%%%%%%%%%%%%%%
\section{Introduction}

The reciprocal $f^*(x)$ of a polynomial $f(x)$ of degree $n$ is defined by $f^*(x)=x^n\,f(\frac{1}{x})$. A polynomial $f(x)$ is called {\it self-reciprocal} if $f^*(x)=f(x)$, i.e. if $f(x)=a_0+a_1x+a_2x^2+\cdots +a_nx^n$, $a_n\neq 0$, is self-reciprocal, then $a_i=a_{n-i}$ for $0\leq i\leq n$. Self-reciprocal polynomials have important applications in coding theory. We briefly explain two applications in the next two paragraphs. 

Let $C$ be a code of length $n$ over $R$, where $R$ is either a ring or a field. Consider the codeword $c=(c_0,c_1,\ldots ,c_{n-2}, c_{n-1})$ in $C$, and denote its reverse by $c^r$ which is given by $c^r=(c_{n-1},c_{n-2},\ldots ,c_1, c_0)$. A code $C$ is  defined to be reversible if $c^r\in C$ for all $c\in C$. If $\tau$ denotes the cyclic shift, then $\tau(c)=(c_{n-1},c_0,\ldots ,c_{n-2})$. A code $C$ is said to be a {\it cyclic code} if the cyclic shift of each codeword is also a codeword. Cyclic codes have a representation in terms of polynomials. For example, the codeword $c=(c_0,c_1,\ldots ,c_{n-1})$ can be represented by the polynomial $f(x)=c_0+c_1x+\cdots c_{n-1}x^{n-1}$ and the cyclic shifts of $c$ correspond to the polynomials $x^if(x)\pmod{x^n-1}$ for $i=0,1, \ldots , n-1$. Among all non-zero codewords in a cyclic code C, there is a unique codeword whose corresponding polynomial $g(x)$ has minimum degree and divides $x^n-1$. The polynomial $g(x)$ is called the generator polynomial of the cyclic code $C$. In \cite{Massey-1964}, Massey studied reversible codes over finite fields and showed that the cyclic code generated by the monic polynomial $g(x)$ is reversible if and only if $g(x)$ is self-reciprocal. 

Deoxyribonucleic acid (DNA) is a molecule that contains all of the information necessary to build and maintain an organism. DNA computing was first introduced by Leonard Adleman when he solved the famous directed Hamiltonian path problem by using DNA molecules as a form of computation; see \cite{Adleman-1994}. Cyclic codes have played a pivotal role in the area of error-correcting codes; see \cite{VP-1998}. The structure of DNA is used as a model for constructing good error correcting codes and conversely error correcting codes that have similar properties with DNA structure are also used to understand DNA itself. The interplay between DNA structure and error correcting codes have been extensively studied by many authors in which self-reciprocal polynomials play a major part. We refer the reader to \cite{{AGZ-2006}, {GG-2013}, {GAS-2013}, {OSY-2014}, {PS-2016}, {YS-2012}, {PAG-2009}} for further details.

 Let $p$ be a prime and $q$ a power of $p$. Let $\Bbb F_q$ be the finite field with $q$ elements. From the definition of a reciprocal polynomial, it is clear that if $f(x)$ is irreducible over $\f_q$, then so is $f^*(x)$. Several authors have surveyed self-reciprocal irreducible monic ({\it srim}) polynomials and obtained many results; see \cite{{HB-1975}, {Meyn-1990}, {Miller-1978}, {Mullen-Yucas-2004}, {PLT-2014}}. 

The present paper is a result of a recent article by the author on the reversed Dickson polynomials of the $(k+1)$-th kind over finite fields in which many previously discovered results on reversed Dickson polynomials were generalized; see \cite{Fernando-2016-3}. For $a\in \f_q$, the $n$-th reversed Dickson polynomial of the $(k+1)$-th kind $D_{n,k}(a,x)$ is defined by
\begin{equation}
D_{n,k}(a,x) = \sum_{i=0}^{\lfloor\frac n2\rfloor}\frac{n-ki}{n-i}\dbinom{n-i}{i}(-x)^{i}a^{n-2i},
\end{equation}

and $D_{0,k}(a,x)=2-k$.

When $p$ is odd, it was shown in \cite{Fernando-2016-3} that the $n$-th reversed Dickson polynomial of the $(k+1)$-th kind $D_{n,k}(1,x)$ can be written as
$$D_{n,k}(1,x)=\Big(\frac{1}{2}\Big)^{n}\,f_{n,k}(1-4x),$$

where 

\begin{equation}\label{E1.2}
f_{n,k}(x)=k\,\,\displaystyle\sum_{j\geq 0} \,\,\binom{n-1}{2j+1}\,\,(x^j-x^{j+1})+2\,\,\displaystyle\sum_{j\geq 0}\,\,\binom{n}{2j}\,\,x^j \,\,\in \mathbb{Z}[x] 
\end{equation}

for $n\geq 1$ and

$$f_{0,k}(x)=2-k.$$

In \cite{Hou-Ly-FFA-2010}, Hou and Ly explored the properties of the reversed Dickson polynomials of the first kind $D_{n,0}(x)$ over finite fields and showed that 
\[
\begin{split}
D_{n,0}(1,x)&=\Big(\frac{1}{2}\Big)^{n-1}\,f_{n}(1-4x);
\end{split}
\]

where 

\begin{equation}\label{E1.3}
f_{n}(x)=\displaystyle\sum_{j\geq 0} \,\,\binom{n}{2j}\,\,x^j.
\end{equation}

In \cite{Hong-Qin-Zhao-FFA-2016-2}, Hong, Qin, and Zhao showed that the reversed Dickson polynomials of the second kind $D_{n,1}(x)$ can be written explicitly as follows. 

\[
\begin{split}
D_{n,1}(1,x)&=\dfrac{1}{2^n}\,f_{n+1}(1-4x),
\end{split}
\]

where 

\begin{equation}\label{E1.4}
f_{n}(x)=\displaystyle\sum_{j\geq 0} \,\,\binom{n}{2j+1}\,\,x^j.
\end{equation}

The author of the present paper showed in \cite{Fernando-2016-1} that the reversed Dickson polynomials of the third kind $D_{n,2}(x)$ can be written explicitly as follows. 

\[
\begin{split}
D_{n,2}(1,x)&=\dfrac{1}{2^{n-1}}\,f_{n}(1-4x),
\end{split}
\]

where 

\begin{equation}\label{E1.5}
f_{n}(x)=\displaystyle\sum_{j\geq 0} \,\,\binom{n}{2j+1}\,\,x^j.
\end{equation}

Note that \eqref{E1.2} is a genelarization of  \eqref{E1.3},  \eqref{E1.4}, and  \eqref{E1.5} for any $k$. The self-reciprocal property of \eqref{E1.3},  \eqref{E1.4}, and  \eqref{E1.5} was used in \cite{Hou-Ly-FFA-2010}, \cite{Hong-Qin-Zhao-FFA-2016-2}, and \cite{Fernando-2016-1}, respectively,  by the aforementioned authors to find necessary conditions for the corresponding reversed Dickson polynomials to be a permutation of $\f_q$. These observations led to the inquisitive question ``when is $f_{n,k}$ a self-reciprocal?''. This paper answers this question completely and it is quite interesting to notice that self-reciprocal polynomials are only arising from the reversed Dickson polynomials of the first and third kinds when $n$ is even, and only from the reversed Dickson polynomials of the second kind with an exception (see Theorem~\ref{T2.4} and Theorem~\ref{T3.4}) when $n$ is odd. An overview of the paper is as follows. 

At the end of this section, we list some preliminaries that will be used in latter sections. 

We explore $f_{n,k}(x) \in \mathbb{Z}[x]$, $f_{n,k}(x) \in \f_p[x]$, where $p$ is an odd prime, $f_{n,k}(x) \in \f_2[x]$ in Section 2, Section 3, and Section 4, respectively,  and find necessary and sufficient conditions for $f_{n,k}(x)$ to be a self-reciprocal. 

In Section 5, we give an introduction to coterm polynomials and their applications in coding theory. Then we obtain several coterm polynomials as a consequence of the self-reciprocal polynomials obtained in Sections 2,3, and 4.

We note to the reader that in Sections 2,3, and 4, we allow the degree $n$ in the definition of a self-reciprocal polynomial to be zero, i.e. we consider non-zero constant polynomials eventhough they are not very interesting. 

\subsection{Preliminaries}

\begin{thm}(Luca's Theorem)\label{T1}
Let $p$ be a prime and let $n, m \geq 0$ be two integers with $p$-adic expansions 
$$n=\alpha_0\,p^0+\alpha_1\,p^1+\alpha_2\,p^2+\cdots +\alpha_t\,p^t,$$
and
$$m=\beta_0\,p^0+\beta_1\,p^1+\beta_2\,p^2+\cdots +\beta_t\,p^t.$$

Then 
$$\binom{n}{m}\equiv \binom{\alpha_0}{\beta_0}\binom{\alpha_1}{\beta_1}\cdots\binom{\alpha_t}{\beta_t} \pmod{p}.$$
\end{thm}

An immediate consequence of the Luca's Theorem is as follows. 

\begin{lem}
$\displaystyle\binom{n}{m}$ is divisible by a prime $p$ if and only if at least one of the base $p$ digits of $m$ is greater than the corresponding digit of $n$. 
\end{lem}

\begin{lem}\label{L1} (See \cite{Meyn-1990})
Every self-reciprocal irreducible polynomial of degree $n\geq 2$ has even degree. 
\end{lem}

%%%%%%%%%%%%%%%%%%%%%%%%%%%%%%%%%%
%   section 2
%%%%%%%%%%%%%%%%%%%%%%%%%%%%%%%%%%

\section{When is $f_{n,k}$ a self-reciprocal?}

In this section we answer the question ``when is $f_{n,k}$ a self-reciprocal?'' by considering two cases: $n$ is odd and $n$ is even. 

Recall that for $n\geq 1$, 
\begin{equation}\label{E2.1}
f_{n,k}(x)=k\,\,\displaystyle\sum_{j\geq 0} \,\,\binom{n-1}{2j+1}\,\,(x^j-x^{j+1})+2\,\,\displaystyle\sum_{j\geq 0}\,\,\binom{n}{2j}\,\,x^j \,\,\in \mathbb{Z}[x], 
\end{equation}

\begin{thm}\label{T2.1}
Let $n>1$ be even. $f_{n,k}(x)$ is a self-reciprocal if and only if $k\in \{0,2\}$. 
\end{thm}

\begin{proof}
Let $k=0$. Then \eqref{E2.1} becomes 
$$f_{n,k}(x)=2\,\,\displaystyle\sum_{j\geq 0}\,\,\binom{n}{2j}\,\,x^j,$$

which is a self-reciprocal since 
$$\binom{n}{2j}=\binom{n}{2(\frac{n}{2}-j)},$$
for $0\leq j\leq \frac{n}{2}$. 

Let $k=2$. Then \eqref{E2.1} becomes
$$f_{n,k}(x)=2\,\,\displaystyle\sum_{j\geq 0} \,\,\binom{n-1}{2j+1}\,\,(x^j-x^{j+1})+2\,\,\displaystyle\sum_{j\geq 0}\,\,\binom{n}{2j}\,\,x^j.$$

So we need to show that 
$$\displaystyle\sum_{j\geq 0} \,\,\binom{n-1}{2j+1}\,\,(x^j-x^{j+1})+\displaystyle\sum_{j\geq 0}\,\,\binom{n}{2j}\,\,x^j$$

is a self-reciprocal. 

Note that 

\begin{equation}\label{NEW1}
\begin{split}
&\displaystyle\sum_{j\geq 0} \,\,\binom{n-1}{2j+1}\,\,(x^j-x^{j+1})+\displaystyle\sum_{j\geq 0}\,\,\binom{n}{2j}\,\,x^j\cr
&=\sum_{j\geq 0}\,\binom{n-1}{2j+1}\,x^{j}-\sum_{j\geq 0}\,\binom{n-1}{2j+1}\,x^{j+1}+\sum_{j\geq 0}\,\binom{n}{2j}\,x^{j}\cr
&=\sum_{j\geq 0}\,\binom{n-1}{2j+1}\,x^{j}-\sum_{j\geq 0}\,\binom{n-1}{2j+1}\,x^{j+1}+\sum_{j\geq 0}\,\binom{n-1}{2j+1}\,x^{j+1}\cr
&+\sum_{j\geq 0}\,\binom{n-1}{2j}\,x^{j}\cr
&=\sum_{j\geq 0}\,\binom{n-1}{2j+1}\,x^{j}+\sum_{j\geq 0}\,\binom{n-1}{2j}\,x^{j}\cr
&=\sum_{j\geq 0}\,\binom{n}{2j+1}\,x^{j}
\end{split}
\end{equation}

$\displaystyle\sum_{j\geq 0}\,\binom{n}{2j+1}\,x^{j}$ is a self-reciprocal since 
$$\binom{n}{2j+1}=\binom{n}{2((\frac{n}{2}-1)-j)+1},$$
for $0\leq j\leq \frac{n}{2}-1$. 

Now assume that 

\begin{equation}\label{E2.2}
\begin{split}
f_{n,k}(x)=k\,\displaystyle\sum_{j\geq 0} \,\,\binom{n-1}{2j+1}\,\,x^j-k\,\displaystyle\sum_{j\geq 0} \,\,\binom{n-1}{2j+1}\,\,x^{j+1}+2\,\displaystyle\sum_{j\geq 0}\,\,\binom{n}{2j}\,\,x^j
\end{split}
\end{equation}
is a self-reciprocal. 

Note that the degree of $f_{n,k}(x)$ is $\dfrac{n}{2}$ and the right hand-side of \eqref{E2.2} can be written as

\begin{equation}\label {E2.3}
(k(n-1)+2)+\,\displaystyle\sum_{j=1}^{\frac{n}{2}-1}\,\Big[\,k\,\binom{n-1}{2j+1}-k\,\binom{n-1}{2j-1}+2\,\binom{n}{2j}\Big]\,x^j+(2-k)\,x^{\frac{n}{2}}.
\end{equation}

If $k\neq 2$, since $f_{n,k}$ is a self-reciprocal, we have 
$$2-k=k(n-1)+2$$
which implies $k=0$. 

If $k\neq 0$, then $k=2$. Otherwise, it contradicts the fact that $n>1$. Note that for $k=2$ and $j=\frac{n}{2}-1$, we have 

$$\,\binom{n-1}{2j+1}-\,\binom{n-1}{2j-1}+\,\binom{n}{2j}\neq 0,$$

and 

\[
\begin{split}
2\,\binom{n-1}{2j+1}-2\,\binom{n-1}{2j-1}+2\,\binom{n}{2j}&=2\Big[\binom{n-1}{2j+1}-\,\binom{n-1}{2j-1}+\,\binom{n}{2j}\Big]\cr
&=2\Big[\binom{n-1}{2j+1} +\binom{n-1}{2j}\Big]\cr
&=2\,\binom{n}{2j+1}\cr
&=2\,\binom{n}{n-1}\cr
&=2n\cr
&=k(n-1)+2.
\end{split}
\]

\end{proof}

\begin{rmk}\label{R1} When $n$ is even, note that in \eqref{E2.3}, if we replace the constant term by the coefficient of $x^{\frac{n}{2}}$ and viceversa, \eqref{E2.3} does not generate self-reciprocal polynomials for any $k$. 

To give an example, let $k=1$ and consider 
$$k\,\binom{n-1}{2j+1}-k\,\binom{n-1}{2j-1}+2\,\binom{n}{2j}$$
in \eqref{E2.3} for $1\leq j\leq \frac{n}{2}-1$. We have

\[
\begin{split}
\binom{n-1}{2j+1}-\binom{n-1}{2j-1}+2\,\binom{n}{2j}&=\binom{n-1}{2j+1}+\binom{n-1}{2j}+\binom{n}{2j}\cr
&=\binom{n}{2j+1}+\binom{n}{2j}\cr
&=\binom{n+1}{2j+1}
\end{split}
\]

Clearly, when $j=1$ and $j=\frac{n}{2}-1$
$$\binom{n+1}{3}\neq \binom{n+1}{n-1}.$$

\end{rmk}

Let's replace the constant term by the coefficient of $x^{\frac{n}{2}}$ in \eqref{E2.3} and define $g_{n,k}$ to be

\begin{equation}\label{G}
g_{n,k}(x):= (2-k)+\,\displaystyle\sum_{j=1}^{\frac{n}{2}-1}\,\Big[\,k\,\binom{n-1}{2j+1}-k\,\binom{n-1}{2j-1}+2\,\binom{n}{2j}\Big]\,x^j+(2-k)\,x^{\frac{n}{2}}.
\end{equation}

Also, replace the coefficient of $x^{\frac{n}{2}}$ by the constant term in \eqref{E2.3} and define $h_{n,k}$ to be

$$h_{n,k}(x):= (k(n-1)+2)+\,\displaystyle\sum_{j=1}^{\frac{n}{2}-1}\,\Big[\,k\,\binom{n-1}{2j+1}-k\,\binom{n-1}{2j-1}+2\,\binom{n}{2j}\Big]\,x^j+(k(n-1)+2)\,x^{\frac{n}{2}}. $$

Then we have the following result. 

\begin{thm}\label{T2.3}
Let $n>1$ be even. $g_{n,k}$ and $h_{n,k}$ are self-reciprocal if and only if $k=0$
\end{thm}

\begin{proof}

We only need to claim that 

\begin{equation}\label{N1}
k\,\binom{n-1}{2j+1}-k\,\binom{n-1}{2j-1}+2\,\binom{n}{2j}=k\,\binom{n-1}{2(\frac{n}{2}-j)+1}-k\,\binom{n-1}{2(\frac{n}{2}-j)-1}+2\,\binom{n}{2(\frac{n}{2}-j)}
\end{equation}

for $1\leq j\leq \frac{n}{2}-1$ when and only when $k=0$.

When $k=0$, clearly equality holds. 
Now assume that $k\neq 0$. Then from \eqref{N1} we have
\[
k\,\binom{n-1}{2j+1}-k\,\binom{n-1}{2j-1}+2\,\binom{n}{2j}=k\,\binom{n-1}{2(\frac{n}{2}-j)+1}-k\,\binom{n-1}{2(\frac{n}{2}-j)-1}+2\,\binom{n}{2(\frac{n}{2}-j)}
\]

\[
k\,\binom{n-1}{2j+1}-k\,\binom{n-1}{2j-1}=k\,\binom{n-1}{n-2j+1}-k\,\binom{n-1}{n-2j-1}
\]

\[
k\,\Big[\binom{n-1}{2j+1}+\binom{n-1}{2j}\Big]\,=\,k\,\Big[\binom{n-1}{2j-1}+\binom{n-1}{2j-2}\Big]
\]

which implies 

\[
\binom{n}{2j+1}\,=\,\binom{n}{2j-1} \,\,\textnormal{for}\,\, 1\leq j\leq \frac{n}{2}-1,
\]

which is a contradiction. 

\end{proof}

\begin{thm}\label{T2.4}
Let $n>1$ be odd. $f_{n,k}(x)$ is a self-reciprocal if and only if $k=1$ or $n=3$ when $k=3$. 
\end{thm}

\begin{proof}

Let $k=1$ in \eqref{E2.1}. Then we have 

\[
\begin{split}
f_{n,k}(x)&=\displaystyle\sum_{j\geq 0} \,\,\binom{n-1}{2j+1}\,\,x^j-\displaystyle\sum_{j\geq 0} \,\,\binom{n-1}{2j+1}\,\,x^{j+1}+2\,\displaystyle\sum_{j\geq 0}\,\,\binom{n}{2j}\,\,x^j\cr
&=\sum_{j\geq 0}\,\binom{n}{2j+1}\,x^{j}+ \displaystyle\sum_{j\geq 0}\,\,\binom{n}{2j}\,\,x^j\cr
&=\sum_{j\geq 0}\,\binom{n+1}{2j+1}\,x^{j}
\end{split}
\]

Clearly, $\displaystyle\sum_{j\geq 0}\,\binom{n+1}{2j+1}\,x^{j}$ is a self-reciprocal polynomial since 

$$\binom{n+1}{2j+1}=\binom{n+1}{2(\frac{n-1}{2}-j)+1},$$

for $0\leq j\leq \frac{n-1}{2}$. 

Now assume that $f_{n,k}(x)$ is a self-reciprocal polynomial. Note that the degree of $f_{n,k}(x)$ is $\dfrac{n-1}{2}$ and the right hand-side of \eqref{E2.2} can be written as

\begin{equation}\label {E2.4}
(k(n-1)+2)+\,\displaystyle\sum_{j=1}^{\frac{n-1}{2}-1}\,\Big[\,k\,\binom{n-1}{2j+1}-k\,\binom{n-1}{2j-1}+2\,\binom{n}{2j}\Big]\,x^j+(-k(n-1)+2n)\,\,x^{\frac{n-1}{2}}.
\end{equation}

$k(n-1)+2=-k(n-1)+2n$ implies $k=1$. 

When $k\neq 1$, If $-k(n-1)+2n=0$, then $k\neq 0, 2$. 

Now assume that $k\neq 0,1,2$. Then$-k(n-1)+2n=0$ if and only if $n=3$ when $k=3$ since $n$ is odd and 

$$n=\frac{k}{k-2}\in \mathbb{Z}\,\,\textnormal{if and only if}\,\, k=3.$$

Note that $f_{n,k}(x)=8$ when $n=3$ and $k=3$.

\end{proof}

\begin{rmk}\label{R11}
Note that when $n=1$, $f_{n,k}(x)=2$ for all $k$. 
\end{rmk}

\begin{rmk}\label{R2}
When $n>1$ is odd, note that in \eqref{E2.4}, if we replace the constant term by the coefficient of $x^{\frac{n-1}{2}}$ and viceversa, \eqref{E2.3} does not generate self-reciprocal polynomials for any $k$. 
\end{rmk}

Let's replace the constant term by the coefficient of $x^{\frac{n-1}{2}}$ in \eqref{E2.4} and define $g_{n,k}^*$ to be
\begin{equation}\label{G*}
\begin{split} 
g_{n,k}^*(x)&:=(-k(n-1)+2n)+\,\displaystyle\sum_{j=1}^{\frac{n-1}{2}-1}\,\Big[\,k\,\binom{n-1}{2j+1}-k\,\binom{n-1}{2j-1}+2\,\binom{n}{2j}\Big]\,x^j\cr
&+(-k(n-1)+2n)\,\,x^{\frac{n-1}{2}}.
\end{split}
\end{equation}

Also, replace the coefficient of $x^{\frac{n-1}{2}}$ by the constant term in \eqref{E2.4} and define $h_{n,k}^*$ to be

$$h_{n,k}^*(x):= (k(n-1)+2)+\,\displaystyle\sum_{j=1}^{\frac{n-1}{2}-1}\,\Big[\,k\,\binom{n-1}{2j+1}-k\,\binom{n-1}{2j-1}+2\,\binom{n}{2j}\Big]\,x^j+(k(n-1)+2)\,\,x^{\frac{n-1}{2}}.$$

Then we have the following result. 

\begin{thm}\label{T2.7}
Let $n>1$ be odd. $g_{n,k}^*$ and $h_{n,k}^*$ are self-reciprocal if and only if $k=1$
\end{thm}

\begin{proof}

We only need to claim that 

\begin{equation}\label{N2}
k\,\binom{n-1}{2j+1}-k\,\binom{n-1}{2j-1}+2\,\binom{n}{2j}=k\,\binom{n-1}{2(\frac{n-1}{2}-j)+1}-k\,\binom{n-1}{2(\frac{n-1}{2}-j)-1}+2\,\binom{n}{2(\frac{n-1}{2}-j)}
\end{equation}

for $1\leq j\leq \frac{n-1}{2}-1$ when and only when $k=1$.

Let $k=1$. Then from the left hand side of \eqref{N2}, we have

\[
\begin{split}
\binom{n-1}{2j+1}-\binom{n-1}{2j-1}+2\,\binom{n}{2j}&= \binom{n-1}{2j+1}+\binom{n-1}{2j}+2\,\binom{n}{2j}\cr
&=\binom{n}{2j+1}+\binom{n}{2j}\cr
&=\binom{n+1}{2j+1}\cr
&=\binom{n+1}{n-2j}\cr
&=\binom{n+1}{n-1-2j+1}\cr
&=\binom{n+1}{2(\frac{n-1}{2}-j)+1}\cr
&=\binom{n}{2(\frac{n-1}{2}-j)+1}+\binom{n}{2(\frac{n-1}{2}-j)}\cr
&=\binom{n-1}{2(\frac{n-1}{2}-j)+1}+\binom{n-1}{2(\frac{n-1}{2}-j)}+\binom{n}{2(\frac{n-1}{2}-j)}\cr
&=\binom{n-1}{2(\frac{n-1}{2}-j)+1}-\binom{n-1}{2(\frac{n-1}{2}-j)-1}+2\,\binom{n}{2(\frac{n-1}{2}-j)}, 
\end{split}
\]

which is the right hand side of \eqref{N2}.

Now assume that $k\neq 1$. Then from \eqref{N2} we have
\[
k\,\binom{n-1}{2j+1}-k\,\binom{n-1}{2j-1}+2\,\binom{n}{2j}=k\,\binom{n-1}{2(\frac{n-1}{2}-j)+1}-k\,\binom{n-1}{2(\frac{n-1}{2}-j)-1}+2\,\binom{n}{2(\frac{n-1}{2}-j)}
\]

\[
k\,\binom{n-1}{2j+1}-k\,\binom{n-1}{2j-1}+2\,\binom{n}{2j}=k\,\binom{n-1}{n-2j}-k\,\binom{n-1}{n-2j-2}+2\,\binom{n}{n-1-2j}
\]

\[
k\,\binom{n-1}{2j+1}-k\,\binom{n-1}{2j-1}+2\,\binom{n}{2j}=k\,\binom{n-1}{2j-1}-k\,\binom{n-1}{2j+1}+2\,\binom{n}{2j+1}
\]

\[
k\,\binom{n-1}{2j+1}-k\,\binom{n-1}{2j-1}+\binom{n}{2j}-\binom{n}{2j+1}=0
\]

\[
(k-1)\,\binom{n-1}{2j+1}+(1-k)\,\binom{n-1}{2j-1}=0
\]
which implies 

\[
\binom{n-1}{2j+1}\,=\,\binom{n-1}{2j-1} \,\,\textnormal{for}\,\, 1\leq j\leq \frac{n-1}{2}-1,
\]

which is a contradiction. 

\end{proof}

%%%%%%%%%%%%%%%%%%%%%%%%%%%%%%%%%%
%   section 3
%%%%%%%%%%%%%%%%%%%%%%%%%%%%%%%%%%

\section{In Odd Characteristic}

Let $n>1$, $p$ be an odd prime, and $0\leq k\leq p-1$. Consider 

\begin{equation}\label{E3.1}
f_{n,k}(x)=k\,\,\displaystyle\sum_{j\geq 0} \,\,\binom{n-1}{2j+1}\,\,(x^j-x^{j+1})+2\,\,\displaystyle\sum_{j\geq 0}\,\,\binom{n}{2j}\,\,x^j \,\,\in \f_p[x].
\end{equation}

\begin{thm} \label{T3.1}
Assume that $n$ is even. Then $f_{n,k}(x)$ is a self-reciprocal if and only if one of the following holds 

\begin{itemize}
\item[(i)] $k=0$.
\item[(ii)] $k=2$ and $n\neq (2l)p$, where $l\in \mathbb{Z}^+$.
\end{itemize}

\end{thm}

\begin{proof}
Let $k=0$. Then 

$$f_{n,k}(x)=2\,\,\displaystyle\sum_{j\geq 0}\,\,\binom{n}{2j}\,\,x^j.$$

We claim that

$$\binom{n}{2j}\equiv \binom{n}{2(\frac{n}{2}-j)}\pmod{p},$$
for $0\leq j\leq \frac{n}{2}$. 

Consider the $p$-adic expansions 

$$n=\alpha_0\,p^0+\alpha_1\,p^1+\alpha_2\,p^2+\cdots +\alpha_t\,p^t,$$

and

$$2j=\beta_0\,p^0+\beta_1\,p^1+\beta_2\,p^2+\cdots +\beta_t\,p^t.$$

Then by Luca's theorem, we have 

\[
\begin{split}
\binom{n}{2j}&\equiv \binom{\alpha_0}{\beta_0}\binom{\alpha_1}{\beta_1}\binom{\alpha_2}{\beta_2}\cdots\binom{\alpha_t}{\beta_t} \pmod{p}\cr
&= \binom{\alpha_0}{\alpha_0-\beta_0}\binom{\alpha_1}{\alpha_1-\beta_1}\binom{\alpha_2}{\alpha_2-\beta_2}\cdots \binom{\alpha_t}{\alpha_t-\beta_t}\cr
&\equiv \binom{n}{2(\frac{n}{2}-j)}\pmod{p}
\end{split}
\]

Then the claim follows from the fact that 

\begin{center}
$\binom{n}{2j}\equiv 0\pmod{p}$ if and only if  there exists an $0\leq l\leq t$ such that $\beta_{l}>\alpha_{l}$ if and only if $\binom{n}{2(\frac{n}{2}-j)}\equiv 0\pmod{p}$. 
\end{center}

Let $k=2$ and $n\neq (2l)p$, where $l\in \mathbb{Z}^+$. Note that $n\neq (2l)p$ implies $p\not\vert (2n)$. 

From \eqref{NEW1}, we have

\begin{equation}\label{NEW27}
\begin{split}
f_{n,k}(x)&=2\,\,\displaystyle\sum_{j\geq 0} \,\,\binom{n-1}{2j+1}\,\,(x^j-x^{j+1})+2\,\,\displaystyle\sum_{j\geq 0}\,\,\binom{n}{2j}\,\,x^j \cr
&=2\,\,\displaystyle\sum_{j\geq 0} \,\,\binom{n}{2j+1}\,\,x^j\cr
\end{split}
\end{equation}

Consider the $p$-adic expansions 

$$n=\alpha_0\,p^0+\alpha_1\,p^1+\alpha_2\,p^2+\cdots +\alpha_t\,p^t,$$

and

$$2j+1=\beta_0\,p^0+\beta_1\,p^1+\beta_2\,p^2+\cdots +\beta_t\,p^t.$$

Then by Luca's theorem, we have 

\[
\begin{split}
\binom{n}{2j+1}&\equiv \binom{\alpha_0}{\beta_0}\binom{\alpha_1}{\beta_1}\binom{\alpha_2}{\beta_2}\cdots\binom{\alpha_t}{\beta_t} \pmod{p}\cr
&= \binom{\alpha_0}{\alpha_0-\beta_0}\binom{\alpha_1}{\alpha_1-\beta_1}\binom{\alpha_2}{\alpha_2-\beta_2}\cdots \binom{\alpha_t}{\alpha_t-\beta_t}\cr
&\equiv \binom{n}{2((\frac{n}{2}-1)-j)+1}\pmod{p}
\end{split}
\]

Then the claim follows from the fact that 

\begin{center}
$\binom{n}{2j+1}\equiv 0\pmod{p}$ if and only if  there exists an $0\leq l\leq t$ such that $\beta_{l}>\alpha_{l}$ if and only if $\binom{n}{2((\frac{n}{2}-1)-j)+1}\equiv 0\pmod{p}$. 
\end{center}

Now assume that $f_{n,k}(x)$ is self-reciprocal. From \eqref{E2.3}, we have

\begin{equation}\label{E3.2}
(k(n-1)+2)+\,\displaystyle\sum_{j=1}^{\frac{n}{2}-1}\,\Big[\,k\,\binom{n-1}{2j+1}-k\,\binom{n-1}{2j-1}+2\,\binom{n}{2j}\Big]\,x^j+(2-k)\,x^{\frac{n}{2}}.
\end{equation}

Since $f_{n,k}(x)$ is self-reciprocal, we have 
$$k(n-1)+2\equiv 2-k \pmod{p}$$
which implies $k=0$ for any even $n$ or $p\vert n$. 

Note that since $n$ is even, $\frac{n}{p}$ is even, so we can write $n=(2l)p$ for some $l\in \mathbb{Z}^+$. 

Now we claim that when $k\neq 0, 2$ and $p\vert n$, it contradicts our assumption that $f_{n,k}(x)$ is self-reciprocal.

Let $k\neq 0,2$ and $p\vert n$. Then from \eqref{E3.2} we have 

\begin{equation}\label{E3.3}
(2-k)+\,\displaystyle\sum_{j=1}^{\frac{n}{2}-1}\,\Big[\,k\,\binom{n-1}{2j+1}-k\,\binom{n-1}{2j-1}+2\,\binom{n}{2j}\Big]\,x^j+(2-k)\,x^{\frac{n}{2}}.
\end{equation}

Now we claim that

\begin{equation}\label{E3.4}
\begin{split}
k\,\binom{n-1}{2j+1}-k\,\binom{n-1}{2j-1}+2\,\binom{n}{2j}&\not \equiv k\,\binom{n-1}{2(\frac{n}{2}-j)+1}-k\,\binom{n-1}{2(\frac{n}{2}-j)-1}+2\,\binom{n}{2(\frac{n}{2}-j)}\cr
\pmod{p}
\end{split}
\end{equation}

for $1\leq j\leq \frac{n}{2}-1$.  

Let 
\[
\begin{split}
k\,\binom{n-1}{2j+1}-k\,\binom{n-1}{2j-1}+2\,\binom{n}{2j}&\equiv k\,\binom{n-1}{2(\frac{n}{2}-j)+1}-k\,\binom{n-1}{2(\frac{n}{2}-j)-1}+2\,\binom{n}{2(\frac{n}{2}-j)}\cr
\pmod{p}
\end{split}
\]

Then by the proof of Theorem~\ref{T2.3}, we have

\[
\binom{n}{2j+1}\,\equiv\,\binom{n}{2j-1} \pmod{p}, 
\]

for $1\leq j\leq \frac{n}{2}-1$, which is clearly a contradiction. 

Now let $k=2$ in \eqref{E3.2}. Then we have

\begin{equation}\label{E3.5}
2n+\,\displaystyle\sum_{j=1}^{\frac{n}{2}-1}\,\Big[\,2\,\binom{n-1}{2j+1}-2\,\binom{n-1}{2j-1}+2\,\binom{n}{2j}\Big]\,x^j. 
\end{equation}

We show that when $k=2$, if $f_{n,k}$ is a self-reciprocal, then $n\neq p(2l)$, i.e. $p\not\vert n$.

When $k=2$, Assume that $n=p(2l)$, i.e. $p\vert n$. Then the constant term in \eqref{E3.5} vanishes, and as a result $f_{n,k}$ is not a self-reciprocal. Hence the proof. 

Note that \eqref{E3.5} can also be written as

\begin{equation}\label{E3.6}
2n+2\,\displaystyle\sum_{j=1}^{\frac{n}{2}-1}\,\,\binom{n}{2j+1}\,\,x^j=2\,\displaystyle\sum_{j=0}^{\frac{n}{2}-1}\,\,\binom{n}{2j+1}\,\,x^j
\end{equation}

\end{proof}

\begin{cor}
If $k=0$ and $n>2$ with $n\equiv 2 \pmod{4}$, then $f_{n,k}(x)$ is not an irreducible self-reciprocal polynomial. 
\end{cor}

\begin{proof}
If $n\equiv 2 \pmod{4}$, then the degree of $f_{n,k}(x)$ is odd. The rest of the proof follows from Theorem~\ref{T3.1} (i) and Lemma~\ref{L1}. 
\end{proof}

\begin{cor}
If $k=2$ and $n\neq (2l)p$ with $n\equiv 0 \pmod{4}$, where $l\in \mathbb{Z}^+$, then $f_{n,k}(x)$ is not an irreducible self-reciprocal polynomial. 
\end{cor}

\begin{proof}
If $n\equiv 0 \pmod{4}$, then the degree of $f_{n,k}(x)$ is odd when $k=2$. The rest of the proof follows from Theorem~\ref{T3.1} (i) and Lemma~\ref{L1}. 
\end{proof}

\begin{thm} \label{T3.4}
Assume that $n>0$ is odd. Then $f_{n,k}(x)$ is a self-reciprocal if and only if one of the following holds 

\begin{itemize}
\item[(i)] $n=1$ for any $k$.
\item[(ii)] $k=0$ and $n=p^l$, where $l\in \mathbb{Z}^+$. 
\item[(iii)] $n=3$ and $k=3$ when $p>3$. 
\item[(iv)] $k=1$ and $n+1\neq (2l)p$, where $l\in \mathbb{Z}^+$. 
\end{itemize}

\end{thm}

\begin{proof}

From Remark~\ref{R11}, we have $f_{1,k}(x)=2$ for all $k$. When $p>3$, $f_{3,3}(x)=8$. 

Let $k=0$ and $n=p^l$ in \eqref{E3.1}, where $l\in \mathbb{Z}^+$. Then 

$$f_{n,k}(x)=2\,\,\displaystyle\sum_{j\geq 0}\,\,\binom{n}{2j}\,\,x^j=2.$$

Now let $k=1$ in \eqref{E3.1} and assume that  $n+1\neq (2l)p$, where $l\in \mathbb{Z}^+$. Then from Theorem~\ref{T2.4} we have 

\[
\begin{split}
f_{n,k}(x)&=\sum_{j\geq 0}\,\binom{n+1}{2j+1}\,x^{j}
\end{split}
\] 

Note that $n+1\neq (2l)p$ implies $p\not\vert (n+1)$. 

Consider the $p$-adic expansions 

$$n+1=\alpha_0\,p^0+\alpha_1\,p^1+\alpha_2\,p^2+\cdots +\alpha_t\,p^t,$$

and

$$2j+1=\beta_0\,p^0+\beta_1\,p^1+\beta_2\,p^2+\cdots +\beta_t\,p^t.$$

Then by Luca's theorem, we have 

\[
\begin{split}
\binom{n+1}{2j+1}&\equiv \binom{\alpha_0}{\beta_0}\binom{\alpha_1}{\beta_1}\binom{\alpha_2}{\beta_2}\cdots\binom{\alpha_t}{\beta_t} \pmod{p}\cr
&= \binom{\alpha_0}{\alpha_0-\beta_0}\binom{\alpha_1}{\alpha_1-\beta_1}\binom{\alpha_2}{\alpha_2-\beta_2}\cdots \binom{\alpha_t}{\alpha_t-\beta_t}\cr
&\equiv \binom{n+1}{2(\frac{n-1}{2}-j)+1}\pmod{p}
\end{split}
\]

Then the claim follows from the fact that 

\begin{center}
$\binom{n+1}{2j+1}\equiv 0\pmod{p}$ if and only if  there exists an $0\leq l\leq t$ such that $\beta_{l}>\alpha_{l}$ if and only if $ \binom{n+1}{2(\frac{n-1}{2}-j)+1}\equiv 0\pmod{p}$. 
\end{center}

Now assume that $f_{n,k}(x)$ is a self-reciprocal. From \eqref{E2.4} we have

\begin{equation}\label{E3.7}
(k(n-1)+2)+\,\displaystyle\sum_{j=1}^{\frac{n-1}{2}-1}\,\Big[\,k\,\binom{n-1}{2j+1}-k\,\binom{n-1}{2j-1}+2\,\binom{n}{2j}\Big]\,x^j+(-k(n-1)+2n)\,\,x^{\frac{n-1}{2}}.
\end{equation}

\textbf{Case 1}. Here we consider the case where the constant term and the coefficient of $x^{\frac{n-1}{2}}$ being non-zero. 

Since $f_{n,k}$ is a self-reciprocal, we have 
$$k(n-1)+2\equiv -k(n-1)+2n \pmod{p},$$

which implies 

$$(k-1)(n-1)\equiv 0 \pmod{p}.$$

Since $n$ is odd, we have $n=1$ for any $k\neq 1$. 

Let $k=1$ in \eqref{E3.7} to obtain

\begin{equation}\label{E3.8}
\begin{split}
&(n+1)+\,\displaystyle\sum_{j=1}^{\frac{n-1}{2}-1}\,\Big[\binom{n-1}{2j+1}-\binom{n-1}{2j-1}+2\,\binom{n}{2j}\Big]\,x^j+(n+1)\,\,x^{\frac{n-1}{2}}\cr
&=\displaystyle\sum_{j=0}^{\frac{n-1}{2}}\,\binom{n+1}{2j+1}\,x^j.
\end{split}
\end{equation}

Since \eqref{E3.8} is a self-reciprocal, we have either $n=1$, in which case $f_{n,k}(x)$ is a constant polynomial, or $p\not \vert (n+1)$ which implies $n+1\neq (2l)p$, where $l\in \mathbb{Z}^+$. Hence we have (i) and (iv). 

\textbf{Case 2}. Here we consider the case where the constant term is non-zero, but coefficient of $x^{\frac{n-1}{2}}$ is zero in \eqref{E3.7}, i.e.

\begin{equation}\label{E3.9}
k(n-1)+2\not\equiv 0 \pmod{p}\,\,\,\,\textnormal{and}\,\,\,\,-k(n-1)+2n\equiv 0\pmod{p}.
\end{equation}

$-k(n-1)+2n\equiv 0\pmod{p}$ implies $n\not\equiv1\pmod{p}$ and $k\neq 2$.  

From $-k(n-1)+2n\equiv 0\pmod{p}$ we have 

\begin{equation}\label{E3.10}
n\equiv \frac{k}{k-2} \pmod{p},
\end{equation}

which implies 

\begin{itemize}
\item [(a)] $k=0$ for any $p$, 
\item [(b)] $k=3$ when $p>3$,  
\item [(c)] $k=1$ for any $p$, or 
\item [(d)] $k=4$ when $p>3$.
\end{itemize}

If $k=0$, since $n$ is an odd, we have $n=p^l$, where $l\in \mathbb{Z}^+$. Note that in this case $f_{n,k}(x)=2$. 

If $k=3$ and $p>3$, we have $n\equiv 3\pmod{p}$. Since $n$ is odd, $n=3$. Note that in this case $f_{n,k}(x)=8$. 

If $k=1$, then it contradicts \eqref{E3.9}. 

The coefficient of $x^{\frac{n-1}{2}-1}$ in \eqref{E3.7} is 

$$k\,\binom{n-1}{1}-k\,\binom{n-1}{3}+2\,\binom{n}{3}.$$

Assume that $k=4$ when $p>3$. When $k=4$, from \eqref{E3.10}, $n\equiv 2 \pmod{p}$. Then the coefficient of $x^{\frac{n-1}{2}-1}$ is

\[
\begin{split}
4\,\binom{n-1}{1}-4\,\binom{n-1}{3}+2\,\binom{n}{3}&=4\,\binom{n}{2}-2\,\binom{n}{3}\equiv 4 \pmod{p}
\end{split}
\]

But the constant term in \eqref{E3.7} is 

$$k(n-1)+2=4(n-1)+2\equiv 6 \pmod{p}.$$

This contradicts our assumption that $f_{n,k}(x)$ is a self-reciprocal. 

Hence the proof. 

\end{proof}

\begin{cor}
If $k=1$ and $n+1\neq (2l)p$ with $n\equiv 3 \pmod{4}$, where $l\in \mathbb{Z}^+$, then $f_{n,k}(x)$ is not an irreducible self-reciprocal polynomial. 
\end{cor}

\begin{proof}
If $n\equiv 3 \pmod{4}$, then the degree of $f_{n,k}(x)$ is odd when $k=1$. The rest of the proof follows from Theorem~\ref{T3.4} (iv) and Lemma~\ref{L1}. 
\end{proof}

%%%%%%%%%%%%%%%%%%%%%%%%%%%%%%%%%%
%   section 4
%%%%%%%%%%%%%%%%%%%%%%%%%%%%%%%%%%

\section{In Characteristic 2}

When $p=2$, \eqref{E2.1} becomes 

\begin{equation}\label{E4.1}
f_{n,k}(x)=k\,\,\displaystyle\sum_{j\geq 0} \,\,\binom{n-1}{2j+1}\,\,(x^j-x^{j+1}) \,\,\in \f_2[x].
\end{equation}

Note that we only need to consider $k=1$. We note to the reader that when $p=2$, the polynomials defined by $f_{n,k}$ do not arise from reversed Dickson polynomials. 

\begin{thm}
Let $n>1$ and $k=1$. Then $f_{n,k}(x)$ is a self-reciprocal if and only if $n$ is even. 
\end{thm}

\begin{proof}
Necessity immediately follows from the fact that $\binom{n-1}{2j+1}=0$ when $n$ is odd. 
For the sufficiency, assume that $n$ is even. Then from \eqref{E4.1} we have

\begin{equation}\label{E4.2}
\begin{split}
f_{n,k}(x)&=\displaystyle\sum_{j\geq 0} \,\,\binom{n-1}{2j+1}\,\,x^j+\displaystyle\sum_{j\geq 0} \,\,\binom{n-1}{2j+1}\,\,x^{j+1}\cr
&=1+\displaystyle\sum_{j=1}^{\frac{n}{2}-1}\Big[\binom{n-1}{2j+1}+\binom{n-1}{2j-1}\Big]\,x^j+x^{\frac{n}{2}}
\end{split}
\end{equation}

Note that 

\[
\begin{split}
\binom{n-1}{2j+1}+\binom{n-1}{2j-1}&=\binom{n-1}{2j+1}+\binom{n-1}{2j}+\binom{n-1}{2j}+\binom{n-1}{2j-1}\cr
&=\binom{n}{2j+1}+\binom{n}{2j}\cr
&=\binom{n+1}{2j+1}
\end{split}
\]

Then \eqref{E4.2} can be written as 

\begin{equation}\label{E4.3}
\begin{split}
f_{n,k}(x)&=\displaystyle\sum_{j=0}^{\frac{n}{2}}\,\,\binom{n+1}{2j+1}\,\,x^j
\end{split}
\end{equation}

For $0\leq j\leq \frac{n}{2}$, we have 

$$\binom{n+1}{2j+1}=\binom{n+1}{2(\frac{n}{2}-j)+1}.$$

Now we claim that for $1\leq j\leq \frac{n}{2}-1$, If $\binom{n+1}{2j+1}$ vanishes, then so does $\binom{n+1}{2(\frac{n}{2}-j)+1}$.

Consider the $2$-adic expansions 

$$n+1=\alpha_0\,2^0+\alpha_1\,2^1+\alpha_2\,2^2+\cdots +\alpha_t\,2^t,$$

and

$$2j+1=\beta_0\,2^0+\beta_1\,2^1+\beta_2\,2^2+\cdots +\beta_t\,2^t.$$

Note that $\alpha_0=1$ and $\beta_0=1$ since $n+1$ and $2j+1$ are odd. Then 

$$n-2j+1=2^0+(\alpha_1-\beta_1)\,2^1+(\alpha_2-\beta_2)\,2^2+\cdots +(\alpha_t-\beta_t)\,2^t.$$

Then by Luca's theorem, we have 

$$\binom{n+1}{2j+1}\equiv \binom{\alpha_0}{\beta_0}\binom{\alpha_1}{\beta_1}\cdots\binom{\alpha_t}{\beta_t} \pmod{2},$$

and 

$$\binom{n+1}{2(\frac{n}{2}-j)+1}\equiv \binom{\alpha_0}{1}\binom{\alpha_1}{\alpha_1-\beta_1}\cdots \binom{\alpha_t}{\alpha_t-\beta_t} \pmod{2}.$$

Then the claim follows from the fact that 

\begin{center}
$\binom{n+1}{2j+1}\equiv 0\pmod{2}$ if and only if  there exists an $1\leq l\leq t$ such that $\beta_{l}>\alpha_{l}$ if and only if $\binom{n+1}{2(\frac{n}{2}-j)+1}\equiv 0\pmod{2}$. 
\end{center}

\end{proof}

\begin{cor}
If $n>2$ with $n\equiv 2 \pmod{4}$, then $f_{n,k}(x)$ is not an irreducible self-reciprocal polynomial. 
\end{cor}

\begin{proof}
If $n\equiv 2 \pmod{4}$, then the degree of $f_{n,k}(x)$ is odd. The rest of the proof follows from Lemma~\ref{L1}. 
\end{proof}

\begin{rmk}
Note that when $n=2$, $f_{n,k}=x+1$ which is irreducible. 
\end{rmk}

%%%%%%%%%%%%%%%%%%%%%%%%%%%%%%%%%%
%   section 5
%%%%%%%%%%%%%%%%%%%%%%%%%%%%%%%%%%

\section{Coterm Polynomials}

\subsection{Introduction}

Coterm polynomials were introduced by Oztas, Siap, and Yildiz in \cite{OSY-2014}. They studied DNA codes over an extension ring of $\f_2+u\f_2$ with the use of coterm polynomials. 

Let $R$ be a commutative ring with identity. 

\begin{defn}(See \cite{OSY-2014}) Let $f(x)=a_0+a_1x+\cdots +a_{n-1}x^{n-1} \in R[x]/(x^{n}-1)$ be a polynomial, with $a_i\in R$. If for all $1\leq i\leq \lfloor \frac{n}{2} \rfloor$, we have $a_i=a_{n-i}$, then $f(x)$ is said to be a coterm polynomial over $R$. 
\end{defn}

According to the definition $f(x)=a_0+a_1x+\cdots +a_{n-1}x^{n-1}$ is a coterm polynomial in $R[x]/(x^{n}-1)$ if and only if $(a_1,a_2,\ldots ,a_{n-1})$ is self-reversible. 

The classical way of constructing a reversible code is to find a self-reciprocal divisor of $x^n-1$ and construct the cyclic code generated by that divisor. However,  Oztas, Siap, and Yildiz explained a new way to construct reversible codes using coterm polynomials. We refer the reader to  \cite{OSY-2014} and the references therein for further details. 

\subsection{Coterm Polynomials from reversed Dickson polynomials}

Clearly, If $f(x)=a_0+a_1x+a_2x^2+\cdots +a_nx^n$, $a_n\neq 0$, is a self-reciprocal polynmial, then the removal of the term $a_nx^n$ from $f(x)$ gives a coterm polynomial. Using the above fact and the self-reciprocal polynomials obtained in the previous sections, we have the following results. 

Consider 

\begin{equation}
f_{n,k}(x)=k\,\,\displaystyle\sum_{j\geq 0} \,\,\binom{n-1}{2j+1}\,\,(x^j-x^{j+1})+2\,\,\displaystyle\sum_{j\geq 0}\,\,\binom{n}{2j}\,\,x^j \,\,\in \mathbb{Z}[x].
\end{equation}

\begin{thm}\label{T5.1}
Let $n\geq 4$ be even. Define
$$C_{n,k}(x):=f_{n,k}(x)-2x^{\frac{n}{2}}.$$
If $k=0$, then $C_{n,k}(x)$ is a coterm polynomial over $\mathbb{Z}$. 
\end{thm}

\begin{proof}
If $k=0$, then $f_{n,k}(x)$ is a self-reciprocal (see Theorem~\ref{T2.1}). The rest of the proof follows from \eqref{E2.3}. 
\end{proof}

\begin{thm}
Let $n\geq 6$ be even. Define
$$C_{n,k}(x):=f_{n,k}(x)-2n\,x^{\frac{n}{2}-1}.$$
If $k=2$, then $C_{n,k}(x)$ is a coterm polynomial over $\mathbb{Z}$. 
\end{thm}

\begin{proof}
If $k=2$, then $f_{n,k}(x)$ is a self-reciprocal (see Theorem~\ref{T2.1}). The rest of the proof follows from \eqref{E2.3}. 
\end{proof}

\begin{thm}
Let $n\geq 4$ be even. Define
$$C_{n,k}(x):=g_{n,k}(x)-2x^{\frac{n}{2}},$$
where $g_{n,k}(x)$ is the polynomial defined in \eqref{G}. 
If $k=0$, then $C_{n,k}(x)$ is a coterm polynomial over $\mathbb{Z}$. 
\end{thm}

\begin{proof}
If $k=0$, then $f_{n,k}(x)$ is a self-reciprocal (see Theorem~\ref{T2.3}). The rest of the proof follows from \eqref{G}. 
\end{proof}

\begin{thm}
Let $n>3$ be odd. Define
$$C_{n,k}(x):=f_{n,k}(x)-(n+1)x^{\frac{n-1}{2}}.$$
If $k=1$, then $C_{n,k}(x)$ is a coterm polynomial over $\mathbb{Z}$. 
\end{thm}

\begin{proof}
If $k=1$, then $f_{n,k}(x)$ is a self-reciprocal (see Theorem~\ref{T2.4}). The rest of the proof follows from \eqref{E2.4}. 
\end{proof}

\begin{thm}
Let $n> 3$ be odd. Define
$$C_{n,k}(x):=g_{n,k}^*(x)-(n+1)x^{\frac{n-1}{2}},$$
where $g_{n,k}^*(x)$ is the polynomial defined in \eqref{G*}. 
If $k=1$, then $C_{n,k}(x)$ is a coterm polynomial over $\mathbb{Z}$.  
\end{thm}

\begin{proof}
If $k=1$, then $f_{n,k}(x)$ is a self-reciprocal (see Theorem~\ref{T2.7}). The rest of the proof follows from \eqref{G*}. 
\end{proof}

Let's consider 

\begin{equation}
f_{n,k}(x)=k\,\,\displaystyle\sum_{j\geq 0} \,\,\binom{n-1}{2j+1}\,\,(x^j-x^{j+1})+2\,\,\displaystyle\sum_{j\geq 0}\,\,\binom{n}{2j}\,\,x^j \,\,\in \f_p[x],
\end{equation}

where $p$ is an odd prime and $0\leq k\leq p-1$.

\begin{thm}\label{T5.7NEW}
Let $n\geq 4$ be even. Define
$$C_{n,k}(x):=f_{n,k}(x)-2\,x^{\frac{n}{2}}.$$
If $k=0$ and $w_p(n)\neq 2$, where $w_p(n)$ is the base $p$ weight of $n$, then $C_{n,k}(x)$ is a coterm polynomial over $\f_p$. 
\end{thm}

\begin{proof}
If $k=0$, then $f_{n,k}(x)$ is a self-reciprocal (see Theorem~\ref{T3.1}). The rest of the proof follows from \eqref{E3.2}. 
\end{proof}

\begin{rmk}
Let $w_p(n)=2$ in Theorem~\ref{T5.7NEW}. From \eqref{E3.2}, we have 

\begin{equation}\label{E5.3}
\begin{split}
C_{n,k}(x)&=f_{n,k}(x)-2\,x^{\frac{n}{2}}\cr
&=2\,\displaystyle\sum_{j=0}^{\frac{n}{2}-1}\,\binom{n}{2j}\,x^j\cr
&=2+\displaystyle\sum_{j=1}^{\frac{n}{2}-1}\,\binom{n}{2j}\,x^j.
\end{split}
\end{equation}

Assume that $1\leq j\leq \frac{n}{2}-1$, i.e. $2\leq 2j\leq n-2$.

Consider the $p$-adic expansions 

$$n=\alpha_0\,p^0+\alpha_1\,p^1+\alpha_2\,p^2+\cdots +\alpha_t\,p^t,$$

and

$$2j=\beta_0\,p^0+\beta_1\,p^1+\beta_2\,p^2+\cdots +\beta_t\,p^t.$$

Then by Luca's theorem, we have 

\[
\begin{split}
\binom{n}{2j}&\equiv \binom{\alpha_0}{\beta_0}\binom{\alpha_1}{\beta_1}\binom{\alpha_2}{\beta_2}\cdots\binom{\alpha_t}{\beta_t} \pmod{p}
\end{split}
\]

We claim that there exists $0\leq i\leq t$ such that $\alpha_i<\beta_i$ for all $1\leq j\leq \frac{n}{2}-1$.

Assume to the contrary $\alpha_i\geq \beta_i$ for all $0\leq i\leq t$. 

Since $w_p(n)=2$, we have 
$$\alpha_0+\alpha_1+\alpha_2+\cdots +\alpha_t=2.$$

Since $\alpha_i\geq \beta_i$ for all $0\leq i\leq t$, we have 
$$\beta_0+\beta_1+\beta_2+\cdots +\beta_t\leq 2.$$

If $\beta_0+\beta_1+\beta_2+\cdots +\beta_t=0$, then it contradicts the fact that $2j\geq 2$. 

If $\beta_0+\beta_1+\beta_2+\cdots +\beta_t=1$, then there esists an $0\leq i\leq t$ such that $2j=p^i$, a contradiction. 

If $\beta_0+\beta_1+\beta_2+\cdots +\beta_t=2$, then it contradicts the fact that $2j\leq n-2$. 

Hence $$\binom{n}{2j}\equiv 0\pmod{p}\,\,\,\textnormal{for all}\,\,\, 1\leq j\leq \frac{n}{2}-1.$$

From \eqref{E5.3}, we have 
$$C_{n,k}(x)\equiv 2\pmod{p}.$$

\end{rmk}

\begin{thm}\label{T5.7}
Let $n\geq 6$ be even. Define
$$C_{n,k}(x):=f_{n,k}(x)-2n\,x^{\frac{n}{2}-1}.$$
If $k=2$, $n\neq (2l_1)p$, where $l_1\in \mathbb{Z}^+$, and $n\neq p^{l_2}+1$, where $l_2\in \mathbb{Z}^+$, then $C_{n,k}(x)$ is a coterm polynomial over $\f_p$. 
\end{thm}

\begin{proof}
If $k=2$ and $n\neq (2l_1)p$, where $l_1\in \mathbb{Z}^+$, then $f_{n,k}(x)$ is a self-reciprocal (see Theorem~\ref{T3.1}). The rest of the proof follows from \eqref{E3.2}. 
\end{proof}

\begin{rmk}
Let $n=p^{l_2}+1$, where $l_2\in \mathbb{Z}^+$, in Theorem~\ref{T5.7}. From \eqref{NEW27}, we have 
\[
\begin{split}
C_{n,k}(x)&=f_{n,k}(x)-2n\,x^{\frac{n}{2}-1}\cr
&=2\,\displaystyle\sum_{j=0}^{\frac{n}{2}-2}\,\binom{n}{2j+1}\,x^j\cr
&\equiv 2 \pmod{p}.
\end{split}
\]
\end{rmk}

\begin{thm}\label{T5.9}
Let $n>3$ be odd. Define
$$C_{n,k}(x):=f_{n,k}(x)-(n+1)x^{\frac{n-1}{2}}.$$
If $k=1$, $n+1\neq (2l_1)p$, where $l_1\in \mathbb{Z}^+$, and $n\neq p^{l_2}$, where $l_2\in \mathbb{Z}^+$, then $C_{n,k}(x)$ is a coterm polynomial over $\f_p$.
\end{thm}

\begin{proof}
If $k=1$ and $n+1\neq (2l_1)p$, where $l_1\in \mathbb{Z}^+$, then $f_{n,k}(x)$ is a self-reciprocal (see Theorem~\ref{T3.4}). The rest of the proof follows from \eqref{E3.7}. 
\end{proof}

\begin{rmk}
Let $n=p^{l_2}$, where $l_2\in \mathbb{Z}^+$, in Theorem~\ref{T5.9}. From \eqref{E3.7}, we have 
\[
\begin{split}
C_{n,k}(x)&=f_{n,k}(x)-(n+1)x^{\frac{n-1}{2}}\cr
&=\displaystyle\sum_{j=0}^{\frac{n-1}{2}-1}\,\binom{n+1}{2j+1}\,x^j\cr
&\equiv 1 \pmod{p}.
\end{split}
\]
\end{rmk}

\begin{rmk}
In characteristic 2, $f_{n,k}(x)-x^{\frac{n}{2}}$ is a coterm polynomial over $\f_2$ if $n\geq 4$ is even and $n\neq 2^l$, where $l\in \mathbb{Z}^+$. Note that when $n=2^l$, where $l\in \mathbb{Z}^+$, we have $f_{n,k}(x)-x^{\frac{n}{2}}\equiv 1 \pmod{2}$.
\end{rmk}

%%%%%%%%%%%%%%%%%%%%%%%%%%%%%%%%%%
%   section 5
%%%%%%%%%%%%%%%%%%%%%%%%%%%%%%%%%%

\section{Acknowledgements}

The author would like to thank Sartaj Ul Hasan for drawing his attention to \cite{Massey-1964}. The author is also grateful to Boris Tsvelikhovsky for the valuable discussions in Section 2.

%%%%%%%%%%%%%%%%%%%%%%%%%%%%%%%%%%%%%%

\end{document}